\documentclass[11pt,a4paper,fleqn]{scrartcl}    

\usepackage[english]{babel}
\usepackage{amsmath,amsfonts, amsthm,xcolor,amssymb}
\numberwithin{equation}{section}
\usepackage{mathtools, mathabx}
\usepackage[colorlinks=true,linkcolor=blue,citecolor=blue]{hyperref}
\usepackage{graphicx}
\usepackage{comment}
\usepackage[hyperpageref]{backref}

\usepackage{xcolor}

\newtheorem{thm}{Theorem}[section]
\newtheorem{lem}[thm]{Lemma}

\newtheorem{prop}[thm]{Proposition}

\newtheorem{defn}[thm]{Definition}
\theoremstyle{definition}
\newtheorem{rem}[thm]{Remark}
\theoremstyle{remark}

\newcommand{\ds}{\displaystyle}
\newcommand{\norm}[1]{\left\Vert#1\right\Vert}

\newcommand{\abs}[1]{\left\vert#1\right\vert}
\newcommand{\set}[1]{\left\{#1\right\}}
\newcommand{\R}{\mathbb{R}}
\newcommand{\N}{\mathbb{N}}

\newcommand{\de}{\partial}
\newcommand{\eps}{\varepsilon}

\DeclareMathOperator{\dive}{div}
\DeclareMathOperator{\tr}{Tr}

\DeclareMathOperator{\spt}{spt}

\newcommand\restr[2]{{
  \left.\kern-\nulldelimiterspace 
  #1 
  \vphantom{ |} 
  \right|_{#2} 
  }}
\makeatletter 
\def\@makefnmark{} 
\makeatother 
\title{On the first Robin eigenvalue of the Finsler $p$-Laplace operator as $p\to 1$}

\author{Rosa Barbato$^*$, Francesco Della Pietra$^{*}$, Gianpaolo Piscitelli$^{*}$ 
\thanks{Dipartimento di Matematica e Applicazioni ``R. Caccioppoli'', Universit\`a degli studi di Napoli Federico II, Via Cintia, Monte S. Angelo - 80126 Napoli, Italia.  \newline 
Email: f.dellapietra@unina.it (\textit{corresponding author}), gianpaolo.piscitelli@unina.it}
}

\date{}

\begin{document}
\maketitle

\begin{abstract}
 \textbf{Abstract.}
Let $\Omega$ be a bounded, connected, sufficiently smooth open set, $p>1$ and $\beta\in\R$. In this paper, we study the $\Gamma$-convergence, as $p\rightarrow 1^+$, of the functional 
$$J_p(\varphi)=\dfrac{\ds \int_\Omega F^p(\nabla \varphi)dx+\beta\ds\int_{\partial \Omega} \abs{\varphi}^pF(\nu)d\mathcal{H}^{N-1}}{\ds\int_\Omega \abs{\varphi}^pdx}$$
where $\varphi\in W^{1,p}(\Omega)\setminus\{0\}$ and $F$ is a sufficientely smooth norm on $\R^n$. We study the limit of the first eigenvalue $\lambda_1(\Omega,p,\beta)=\inf_{\substack{\varphi\in W^{1,p}(\Omega)\\ \varphi \ne0}}J_p(\varphi)$, as $p\to 1^+$, that
is: 
\begin{equation*}
\Lambda(\Omega,\beta)=\inf_{\substack{\varphi \in BV(\Omega)\\ \varphi\not\equiv 0}}\dfrac{\abs{Du}_F(\Omega)+\min\{\beta,1\}\ds \int_{\partial \Omega}\abs{\varphi}F(\nu)d\mathcal H^{N-1}}{\ds\int_\Omega \abs{\varphi}dx}.
\end{equation*}
Furthermore, for $\beta>-1$, we obtain an isoperimetric inequality for $\Lambda(\Omega,\beta)$ depending on $\beta$.

The proof uses an interior approximation result for $BV(\Omega)$ functions by $C^\infty(\Omega)$ functions in the sense of strict convergence on $\R^n$ and a trace inequality in $BV$ with respect to the anisotropic total variation.

\noindent\textbf{MSC 2020:} 28A75, 35J25, 35P15. \\
\textbf{Keywords and phrases}: Finsler $p$-Laplace eigenvalues; $\Gamma$-convergence; Isoperimetric inequalities; Trace inequalities; Strict interior approximation.
\end{abstract}

\begin{center}
\begin{minipage}{11cm}
\small
\tableofcontents
\end{minipage}
\end{center}

\section{Introduction}
Let $\Omega$ be a bounded, connected, sufficiently smooth open set, $p>1$ and $\beta\in\R$. In this paper, we study the asympthotic behaviour, as $p\to 1^+$, of the following minimum problem
\begin{equation}
    \label{lambda_intro}
\lambda_1(\Omega,p,\beta)=\inf_{\substack{\varphi\in W^{1,p}(\Omega)\\ \varphi \ne0}} J_p(\varphi)
\end{equation}
where
\begin{equation}
\label{functional_Jp_intro}
J_p(\varphi)=\dfrac{\ds \int_\Omega F^p(\nabla \varphi)dx+\beta\ds\int_{\partial \Omega} \abs{\varphi}^pF(\nu)d\mathcal{H}^{N-1}}{\ds\int_\Omega \abs{\varphi}^pdx},
\end{equation}
$\nu$ is the outer normal to $\partial\Omega$ and $F$ is a sufficiently smooth norm on $\R^n$. 
If $u\in W^{1,p}(\Omega)$ is a minimizer of \eqref{lambda_intro}, then it solves the following Robin eigenvalue problem
\begin{equation*}
\begin{cases}
-\mathcal Q_p u=\lambda_1(\Omega,p,\beta)\abs{u}^{p-2}u \;&\textrm{in}\;\Omega\\
F^{p-1}(\nabla u) F_\xi(\nabla u)\cdot \nu +\beta F(\nu)\abs{u}^{p-2}u=0\;&\textrm{on}\;\partial \Omega,
    \end{cases}
\end{equation*}
where $\mathcal Q_p u$ is the anisotropic $p-$Laplace operator 
\[
\mathcal Q_p u:= \dive \left(\frac{1}{p}\nabla_{\xi}[F^{p}](\nabla u)\right).
\] 
From the point of view of finding optimal domains for $\lambda_1(\Omega,p\beta)$, there is a significant difference from the case of $\beta>0$ to the case $\beta<0$. It is  known that the optimal shape with a volume constraint for \eqref{lambda_intro} depends on the sign of $\beta$. For positive values of the Robin parameter, the so-called Wulff shape (see Section \ref{notation_sec} for details) is a minimizer \cite{fragavitone}:
\[
\lambda_1(\Omega,p,\beta)\ge \lambda_1(\mathcal{W},p,\beta)\quad\text{where }|\mathcal{W}|=|\Omega|.
\]
If $\beta<0$, the problem is not completely solved, even in the Euclidean case (that is when $F(\xi)=\sqrt{\sum_i \xi_i^2}$). Indeed, in 1977 Bareket \cite{bareket1977isoperimetric} conjectured that, the first eigenvalue is maximized by a ball in the class of the smooth bounded domains of given volume. In \cite{ferone2015conjectured} it has been showed that it is true for domains close, in certain sense, to a ball. Subsequently, the authors in \cite{freitas2015first} (see \cite{kovavrik2017p} for the $p$-Laplacian case) have disproved the conjecture for $\left|\beta\right|$ large enough and have showed that it is true for small values of $\left|\beta\right|$ in suitable class of domains. In the Finsler setting, this problem has been addressed in \cite{paoli2019two}.

Our final aim is to obtain optimal shapes for the limiting functional of \eqref{lambda_intro}, as $p\to 1^+$. To do that, we first study the limit of $\lambda_1(\Omega,p\beta)$. In particular, we prove that when $\beta>-1$, the functional $J_p$, defined in \eqref{functional_Jp_intro}, $\Gamma-$converges, as $p\to 1^+$, to  
\[
J(\varphi)=\dfrac{\ds \abs{D\varphi}_F(\Omega)+\min\{1,\beta\}\ds\int_{\partial \Omega} \abs{\varphi}F(\nu)d\mathcal{H}^{N-1}}{\ds\int_\Omega \abs{\varphi}dx},
\]
where $|D\varphi|_F(\Omega)$ is the anisotropic total variation of $\varphi$ (see Section \ref{notation_sec} for the precise definition). This will imply that
\begin{equation}
    \label{Lambda_intro}
\lim_{p\to 1^+}\lambda_1(\Omega,p,\beta)= \Lambda(\beta,\Omega):=\inf_{\varphi\in BV(\Omega)}J(\varphi).
\end{equation}

Then, we prove an isoperimetric inequality for $\Lambda(\Omega,\beta)$. In particular, we obtain that keeping the volume of $\Omega$ fixed, the Wulff shape minimises $\Lambda(\Omega,\beta)$ when $\beta \ge 0$, and maximises it when $-1<\beta < 0$. 

The proof of the convergence result, and then of the two isoperimetric inequalities, relies on two results on the anisotropic total variation, which are also of independent interest. The first one is a trace inequality in the $BV$ space:
\begin{equation}
    \int_{\partial \Omega}\abs{u} F(\nu)\, d\mathcal{H}^{N-1}\leq c_1\abs{Du}_F(\Omega)+c_2\int_{\Omega}\abs{u}dx,\quad \forall u\in BV(\Omega),
\end{equation}
where $c_1$ and $c_2$ are two constants which depend on the geometry of the domain. 
This inequality has been revealed very useful in capillarity problems, and it has been studied for example in \cite{anzellotti1978funzioni,gerhardt,giusti1981equilibrium}.

The second key result is an interior approximation for $BV$ functions by smooth functions with compact support. It is well known that if $\Omega$ is an open set, then the total variation of a function $u\in BV(\Omega)$ can be approximated with the corresponding total variation of a sequence in $C^\infty(\Omega)$. Actually, an analogous result is not true in general if one need to axpproximate $|Du|$ with a sequence of $C^\infty$ function with compact support in $\Omega$. In order to do that, more regularity is needed on $\Omega$. In the Euclidean setting, this problem has been addressed in \cite{littig2014,schmidt2015strict}. In this paper, we show that for any $u\in BV(\Omega)\cap L^p(\Omega)$, for some $p\in [1,\infty)$, there exists a sequence $\{u_k\}_{k\in\mathbb N}\subseteq C_0^\infty(\Omega)$ such that, for any $q\in [1,p]$,
\[
u_k\to u\ \ \text{in}\ \ L^q(\Omega)\quad\text{and}\quad |D u_k|_F(\mathbb R^N)\to |D u|_F(\mathbb R^N).
\]

We finally stress that the problem we deal with is strictly related to capillarity problems. We refer the reader, for example, to \cite{gerhardt,giusti1981equilibrium} for the Euclidean case and to \cite{philippis2015regularity} for the anisotropic case.

The structure of the paper is the following. In Section \ref{notation_sec}, we review some useful tools on the Finsler norm, the anisotropic curvature and functions of bounded variation. In Section \ref{trace_sec} we prove the anisotropic trace inequality for general domains, and for smooth domains. In Section \ref{approx_sec}, we give the strict approximation result and finally, in Section \ref{eig_sec} we prove the $\Gamma-$convergence results and the isoperimetric inequality for $\Lambda(\Omega,\beta)$.

\section{Notation and preliminaries}
\label{notation_sec}
In this Section we give several definitions and properties related the Finsler norm. In particular, we review some basic facts on the anisotropic total variation of a $BV$ function, and on the anisotropic curvatures. 

\subsection{The Finsler norm}
Throughout the paper we will assume that $F$ is a convex, even, $1-$homogeneous function 
\[
\xi\in \R^{N}\mapsto F(\xi)\in [0,+\infty[,
\] 
such that
\begin{equation}
\label{eq:omo}
F(t\xi)=|t|F(\xi), \quad t\in \R,\,\xi \in \R^{N}, 
\end{equation}
 and such that
\begin{equation}
\label{eq:lin}
a|\xi| \le F(\xi),\quad \xi \in \R^{N},
\end{equation}
for some constant $a>0$. It is easily seen that this hypothesis assure the existence of a positive constant $b\ge a$ such that
\[
F(\xi)\le b |\xi|,\quad \xi \in \R^{N}.
\]
Throughout the paper, we will also assume that $F$ belongs to $C^{2}(\mathbb{R}^N\setminus \{0\})$ and that
\begin{equation}
\label{strong}
\nabla^{2}_{\xi}[F^{2}](\xi)\text{ is positive definite in }\R^{N}\setminus\{0\}.
\end{equation}

The assumption \eqref{strong} on $F$ ensures that the operator 
\[
\mathcal Q_p u:= \dive \left(\frac{1}{p}\nabla_{\xi}[F^{p}](\nabla u)\right)
\] 
is elliptic, therefore there exists a positive constant $\gamma$ such that
\begin{equation*}
\sum_{i,j=1}^{n}{\nabla^{2}_{\xi_{i}\xi_{j}}[F^{p}](\eta)
  \xi_i\xi_j}\ge
\gamma |\eta|^{p-2} |\xi|^2 \qquad\forall\eta\in\R^N\setminus\{0\}, \  \forall\xi\in\R^N. \end{equation*}

The polar function $F^o\colon\R^N \rightarrow [0,+\infty[$ of $F$ is
\begin{equation*}
F^o(v)=\sup_{\xi \ne 0} \frac{ \xi\cdot v}{F(\xi)}. 
\end{equation*}
It is easily seen that also $F^o$ is a convex function satisfying the properties \eqref{eq:omo} and
\eqref{eq:lin}. Furthermore, we have
\begin{equation*}
F(v)=\sup_{\xi \ne 0} \frac{ \xi\cdot v}{F^o(\xi)},
\end{equation*}
and from this follows that
\begin{equation}\label{prodscal}
| \xi\cdot \eta | \le F(\xi) F^{o}(\eta) \qquad \forall \xi, \eta \in \R^{N}.
\end{equation}
The Wulff shape centered at the origin is the set denoted by
\[
\mathcal W = \{  \xi \in \R^N \colon F^o(\xi)< 1 \}.
\]

We denote $\kappa_N=|\mathcal W|$, where $|\mathcal W|$ is the Lebesgue measure
of $\mathcal W$. More generally, the set $\mathcal W_r(x_0)$ indicates $r\mathcal W+x_0$, that is the Wulff shape centered at $x_0$ with measure $\kappa_Nr^N$. If no ambiguity occurs, we will write $\mathcal W_r$ instead of $\mathcal W_r(0)$. 

The functions $F$ and $F^o$ enjoy the following properties:
\begin{align}
& F_{\xi}(\xi) \cdot \xi = F(\xi), \quad  F_{\xi}^{o} (\xi)\cdot \xi 
= F^{o}(\xi) &\forall \xi \in
\R^N\setminus \{0\},\\
& F(F_{\xi}^o(\xi))=F^o( F_{\xi}(\xi))=1 &\forall \xi \in
\R^N\setminus \{0\}, \label{FF0xi}
\\
& F^o(\xi)  F_{\xi}(F_{\xi}^o(\xi) ) = F(\xi) 
F_{\xi}^o( F_{\xi}(\xi) ) = \xi &\forall \xi \in
\R^N\setminus \{0\},
\end{align}
where $F_{\xi}=\nabla F(\xi)$. 

Given a bounded domain $\Omega$, the anisotropic distance of $x\in\overline\Omega$ to $\de\Omega$ is defined as
\begin{equation*}
d_{F}(x):= \inf_{y\in \partial \Omega} F^o(x-y), \quad x\in \overline\Omega.
\end{equation*}

We highlight that, when $F(\xi)=\sqrt{\sum_i \xi_i^2}$, then $d_F=d_{\mathcal{E}}$ is the Euclidean distance function from the boundary.

The function $d_{F}$ is a uniform Lipschitz function in $\overline \Omega$, and 
\begin{equation*}
  F(\nabla d_F(x))=1 \quad\text{a.e. in }\Omega.
\end{equation*}
We have that $d_F\in W_{0}^{1,\infty}(\Omega)$. Many properties of the anisotropic distance function are studied in \cite{crasta2007distance}.

Finally, the anisotropic inradius of $\Omega$ is 
\begin{equation*}
R^{F}(\Omega)=\max \{d_{F}(x),\; x\in\overline\Omega\},
\end{equation*}
that is the radius of the largest Wulff shape $\mathcal W_{r}(x)$ contained in $\Omega$.

\subsection{Anisotropic curvatures}
Here we recall some properties of the anisotropic mean curvature, as well as  an integration formula in anisotropic normal coordinates. We refer to \cite{crasta2007distance} for further details.

If $\Omega$ has a $C^2$ boundary, the anisotropic outer normal to $\partial\Omega$ is defined as  
\[
n^F(y)= F_{\xi}(\nu(y)),\qquad y\in \de\Omega,
\]
where $\nu(y)$ is the Euclidean outer normal to $\de \Omega$ at $y$. Moreover, by \eqref{FF0xi} it holds that
\[
F^o(n^F(y))=1.
\]
Let us denote by $T_{y}\de\Omega$ the tangent space to $\de\Omega$ at $y$; the anisotropic Weingarten map is defined as
\[
d n^F\colon T_{y}\de\Omega \to T_{n(y)}\mathcal W.
\]
The eigenvalues $\kappa^{F}_{1}\le \kappa^{F}_{2}\le \ldots \le \kappa_{N-1}^{F}$ of this map are called the anisotropic principal curvatures at $y$ (see also \cite{wang2013sharp}). The anisotropic mean curvature of $\de\Omega$ at a point $y$ is defined as
  \begin{equation*} 
  \mathcal H^{F}(y)= \kappa_{1}^{F}(y)+\ldots \kappa_{N-1}^{F}(y) , \quad y\in
  \de \Omega.
  \end{equation*}

The anisotropic distance $d_F$ is a $C^2$ function in a tubular neighborhood of $\partial\Omega$; hence we are in position to define the matrix-valued function
\[
W(y)=-F_{\xi\xi}(\nabla d_{F}(y))\nabla^{2}d_F(y), \quad y\in\de\Omega.
\]
Based on this function, it is possible to give a different definition of the anisotropic principal curvatures \cite[Remark 5.9]{crasta2007distance}. Since $W(y)v\in T_y\partial\Omega$, for any $v\in\R^n$, it remains defined the map $\overline W(y)\colon T_{y}\to T_{y}$,  as $\overline W(y)w=W(y)w$, $w\in T_{y}$. The matrix $\overline W(y)$ (that is, in general, non-symmetric) admits the real eigenvalues $\kappa_{1}^{F}(y)\le \kappa_{2}^{F}(y)\le \ldots\le \kappa_{N-1}^{F}(y)$. Actually,  the definition is equivalent to the preceding one. Moreover, it holds that
\[  
\mathcal H^{F}(y)= \dive\left[F_{\xi}\left(-{\nabla d_{F}(y)}\right) \right] = \tr(W(y))
\]
(see also \cite[Sec. 3]{wang2013sharp}).

To state the change of variable formula in anisotropic normal coordinates, we need some preliminary definitions.

Let
\[
\Phi(y,t)= y-tF_{\xi}(\nu(y)),\qquad y\in \de\Omega, \quad t\in \R
\]
and for $y\in \de \Omega$,
\[
\ell(y)=\sup\{d_{F}(z),\, z\in\Omega\text{ and }y\in \Pi(z) )\}
\]
where
\begin{equation}
    \label{proiection}
\Pi(z)=\{\eta\in \de\Omega\colon d_{F}(z)=F^{o}(z-\eta)\}
\end{equation}
is the set of the anisotropic projections of a point $z\in\Omega$ on $\de\Omega$.

Then, we recall the following
\begin{thm} (\cite[Theorem 7.1]{crasta2007distance})\label{change_normal}
For every $h\in L^1(\Omega)$, it holds
\[
\int_\Omega h(x)dx=\int_{\partial\Omega} F(\nu(y))\int_0^{\ell(y)}h(\Phi(y,t))J(y,t)\,dt\,d\mathcal H^{N-1}(y),
\]
where \begin{equation}\label{jacobian}
J(y,t)=\ds\prod_{i=1}^{N-1}(1- t \kappa^{F}_i(y)).
\end{equation}
\end{thm}

Since $1- t \kappa^{F}_i(y)>0$ for any $i=1,\ldots,N-1$ (\cite[Lemma 5.4]{crasta2007distance}), $J(y,t)$ is positive. Moreover it holds that
\begin{equation}
\label{laplaciano_cambio}
\frac{-\frac{d}{dt}\left[J(y,t)\right]}{J(y,t)}=\sum_{i=1}^{N-1}\frac{\kappa^{F}_i(y)}{1-t\kappa^{F}_i(y)}.
\end{equation}

Finally, we conclude this section, by recalling that, for any $x\in \Omega$ such that $\Pi(x)=\{y\}$, it holds that (\cite[Lemma 4.3]{crasta2007distance}):
\begin{equation}
\label{dentro_fuori}
\nabla d_F(x)=-\frac{\nu(y)}{F(\nu(y))}.
\end{equation}

\subsection{The anisotropic total variation}
Let $u\in BV(\Omega)$, the total variation of $u$ with respect to $F$ is defined as
\[
|Du|_F(\Omega)=\sup\left\{\int_\Omega u\dive(g)\ dx \ \ : \ g\in C_0^1(\Omega; \R^N), \ F^o(g)\leq 1\right\}
\]
and the perimeter of a set $E$ with respect to $F$ is:
\[
P_F(E;\Omega)=|D\chi_E|_F(\Omega)=\sup\left\{\int_E \dive(g)\ dx \ \ : \ g\in C_0^1(\Omega; \R^N), \ F^o(g)\leq 1\right\}.
\]
Moreover,
\[
P_F(E;\Omega)=\int_{\Omega\cap\partial^* E}F (\nu_E) d\mathcal H^{N-1}
\]
where $\de^*E$ is the reduced boundary of $\Omega$ and $\nu_E$ is the Euclidean normal to $\de E$.

Let us fix $u\in BV(\Omega)$ and assume that $u\equiv 0$ in $\R^N\setminus\Omega$. Then $ u\in BV(\R^N)$ and
\begin{equation}
    \label{salto_variazione}
|D u|_F(\R^N)=|Du|_F(\Omega)+\int_{\partial\Omega}| u|F(\nu)d\mathcal H^{N-1}
\end{equation}
(see for example \cite[Lemma 3.9]{BoundaryTrace}).

For the anisotropic perimeter, an isoperimetric inequality holds. More precisely, 
\begin{equation}
\label{wulffinequality}
    P_F(\Omega)\ge P_F(W_R),
\end{equation}
where $\mathcal W_R$ is the Wulff shape with the same measure of $\Omega$
(see for example \cite[Theorem 2.10]{fonseca1991uniqueness}).

The following approximation results in $BV$ hold (refer to \cite{anzellotti1978funzioni}, \cite[Theorem 1.17]{giusti1984minimal}  for the Euclidean case and to \cite[Proposition 2.1]{aflt} for the Finsler case).
\begin{prop}
\label{approximation_BV}
Let $f \in BV(\Omega)$, then there exists a sequence $\{f_k\}_{k\in\N} \subseteq C^{\infty}(\Omega)$ such that: 
\[\lim_{k\rightarrow+\infty}\int_{\Omega}\abs{f_k-f}dx=0
\] and 
\[
 \lim_{k\rightarrow+\infty}\abs{Df_k}_F(\Omega)=\abs{Df}_F(\Omega).
\]
\end{prop}

\begin{prop}
Let $E$  be a set of finite perimeter in $\Omega$. A sequence of $C^{\infty}$ sets $\{E_k\}_k$ exists, such that: 
\[
\lim_{k \rightarrow +\infty}\int_{\Omega} \abs{\chi_{E_k}-\chi_{E}}dx=0
\]
and
\[
\lim_{k\rightarrow +\infty}|D\chi_{E_k}|_F(\Omega)=P_F(E;\Omega).
\]
\end{prop}

\section{An anisotropic trace inequality}
\label{trace_sec}
In this section we prove a trace inequality in $BV$ with respect to the anisotropic total variation. We first give the result in a general case (Proposition \ref{trace_prop_Q}), then we refine the constants involved in the inequality by requiring more regularity on the boundary of $\Omega$ (Proposition \ref{trace_prop_1}).

Firstly, let us set 
\[
q(y)=\lim_{\rho \rightarrow 0^+}\sup\left\{\dfrac{\ds\int_{\partial\Omega} \chi_A F(\nu) d\mathcal{H}^{N-1}}{\abs{D\chi_A}_F(\Omega)}: A \subset \Omega \cap B_{\rho} (y), |A|>0, P_F(A;\Omega)<+\infty\right\}
\]
and $Q=\sup_{y\in\partial\Omega}q(y)$. The following  inequality generalizes the trace inequality given in \cite[Theorem 4]{anzellotti1978funzioni}.
\begin{prop}
\label{trace_prop_Q}
Let $\Omega$ be a bounded open set with $\mathcal H^{N-1}(\Omega)<+\infty$, and let $u$ be a function in $BV(\Omega)$. Then for any $\varepsilon>0$, it holds
\begin{equation}
    \label{trace_inequality_Q}
\int_{\partial\Omega}|u| F(\nu)d \mathcal H^{N-1}\leq (Q+\varepsilon) |Du|_F(\Omega)+c(\Omega,\varepsilon)\int_\Omega |u| dx,
\end{equation}
where $c(\Omega,\varepsilon)$ does not depend on $u$.
\end{prop}
\begin{proof}

Let us fix $y \in \partial \Omega$ and $\rho (y)>0$ such that 
\[
\int_{\partial\Omega} \chi_{B} F(\nu)d\mathcal{H}^{N-1} \leq (Q+\varepsilon) |D\chi_B|_F(\Omega),
\]
for any $B\subset \Omega\cap B_{\rho (y)}(y)$ with $P_F(B;\Omega)<+\infty$.

If $\spt (u)\subset B_{\rho (y)}(y)$, then
\begin{multline*}
\int_{\partial\Omega}| u| F(\nu)d\mathcal H^{N-1}=\int_{\partial \Omega}F(\nu)\left(\int_0^{+\infty} \chi_{\{\abs{u}>t\}}(y)dt\right) d\mathcal{H}^{N-1}\\
=\int_0^{+\infty}\left(\int_{\partial\Omega}F(\nu)\chi_{\{\abs{u}>t\}}(y)d\mathcal H^{N-1}\right)dt\leq (Q+\varepsilon)\int_0^{+\infty}|D\chi_{\{\abs{u}>t\}}|_F(\Omega) dt.
\end{multline*}
By using the coarea formula, we have
\[
\int_{\partial\Omega}|u|F(\nu)d\mathcal H^{N-1}\leq (Q+\varepsilon)|D|u||_F(\Omega)\leq (Q+\varepsilon)|Du|_F(\Omega).
\]

Let $\{B_{\rho}(y)\}_{y \in \partial \Omega}$ be a cover of $\partial \Omega$ and let us extract a finite sub-cover $B_1,\dots,B_k$. Now, considering a partition of unity $\varphi_1,\dots,\varphi_k$ such that
\[
0 \leq \varphi_i \leq 1, \qquad \varphi_i \in C^1_0(B_i), \qquad \sum_{i=1}^k \varphi_i(y)=1 \qquad \textrm{if}\; y\in \partial \Omega.
\]
If $f \in BV(\Omega)$, then 
\[
\begin{split}
\int_{\partial\Omega} \abs{u} F\left(\nu\right) d\mathcal{H}^{N-1}&\leq (Q+\varepsilon) \left|D\left(\sum_{i=1}^k \varphi_i u\right)\right|_F(\Omega)\\
&\leq (Q+\varepsilon) \sum_{i=1}^k (\abs{\varphi_iDu}_F(\Omega)+\abs{uD\varphi_i}_F(\Omega))\\
&=(Q+\varepsilon) \sum_{i=1}^k\left(\int_\Omega \varphi_id |Du|_F+\int_\Omega ud\abs{D\varphi_i}_F(\Omega)\right)\\
&\leq (Q+\varepsilon)\abs{Du}_F(\Omega)+c(\Omega,\varepsilon)\int_{\Omega}\abs{u}dx.
\end{split}
\]
\end{proof}
If the boundary of $\Omega$ is sufficilently smooth, we can show that $Q$ can be taken equal to $1$ and $\eps=0$. More precisely, we have the following.
\begin{prop}
\label{trace_prop_1}
Let $\Omega$ be a bounded open connected set of class $C^2$. Then there exists a positive constant $c$ such that
\begin{equation}
    \label{trace_inequality_1}
    \int_{\partial \Omega}\abs{u} F(\nu)\, d\mathcal{H}^{N-1}\leq \abs{Du}_F(\Omega)+c\int_{\Omega}\abs{u}dx,\quad \forall u\in BV(\Omega).
\end{equation}
\end{prop}
\begin{proof}
Since $\Omega$ is $C^2$, then a uniform sphere condition of radius $r>0$ holds, in the sense that for every point $y\in\de\Omega$ there exists $z\in\Omega$ such that $y\in \overline{B_r(z)}\subset\overline\Omega$. Let $R\in]0,+\infty[$ be the maximum of the principal radii of curvature of $\de \mathcal W$. Such maximum exists being $F$ (and $F^o$) strongly convex. If   $\kappa_1^F,\dots,\kappa_{N-1}^F$ are the anisotropic principal curvatures, we have that 
\[
{\kappa_i}^F(y)\leq \dfrac{1}{\mu},  \quad i=1,\ldots, N-1,
\]
with $\mu=\frac{r}{R}$ (\cite[Lemma 5.4]{crasta2007distance}). Therefore, in the set $\Omega_{\frac\mu 2}:=\left\{x\in\Omega\ : \ d_F(x)< \frac{\mu}2\right\}$, it holds that $d_F$ is $C^2$ (\cite[Lemma 4.1 and Theorem 4.16]{crasta2007distance}) and
\[
\dfrac{{\kappa^F_i(y)}}{1-{\kappa^F_i(y)}d_F(x)}\le
\begin{cases}
0 &\text{ if } {\kappa^F_i}\le 0\\
\dfrac{2}{\mu} & \text{ if }0<{\kappa^F_i}\le \dfrac{1}{\mu},
\end{cases}
\]
where $y\in\de\Omega$ is the anisotropic projection of $x\in \Omega$ on $\de\Omega$. 
Then
\begin{equation}
    \label{boundmc}
    \sum_{i=1}^{N-1}\dfrac{{\kappa_i}^F(y)}{1-{\kappa_i}^F(y)d_F(x)} \le 2\frac{N-1}{\mu}
\end{equation}
for $x\in \Omega_{\frac{\mu}2}$.
We may restrict ourselves to the case $u$ is nonnegative and smooth. Integrating by parts and recalling that $F(\nabla d_F(x))=1$ in $\Omega$, it holds that
\begin{multline*}
\int_{\Omega} -\Delta_F d_F(x)\, u(x) \left(\frac{\mu}{2}-d_F(x)\right)^+\,dx\\ 
=\int_{\Omega}F_{\xi}(\nabla d_F(x) )\cdot \nabla u(x) \left(\frac{\mu}{2}-d_F(x)\right)^+\, dx\\
\quad - \int_{\Omega_{\frac\mu 2}} u(x) F_{\xi}(\nabla d_F(x) )\cdot  \nabla d_F(x)\, dx -\frac{\mu}{2}\int_{\partial\Omega} u(x)   F_{\xi}(\nabla d_F(x))\cdot \nu\, dx\\
=\int_{\Omega}F_{\xi}(\nabla d_F(x))\cdot\nabla u(x) \left(\frac{\mu}{2}-d_F(x)\right)^+dx-\int_{\Omega_{\frac{\mu}{2}}} u(x)\,dx +\frac{\mu}{2}\int_{\partial \Omega} u(y)\, F(\nu)d\mathcal H^{N-1}. 
\end{multline*}
Now, we estimate the term
\[
\int_{\Omega}F_{\xi}(\nabla d_F(x))\cdot \nabla u (x) \left(\frac{\mu}{2}-d_F(x)\right)^+dx.
\]
By \eqref{prodscal} and \eqref{FF0xi} it holds that
\begin{equation}
\label{passtracciain}
\int_{\Omega}F_{\xi}(\nabla d_F(x))\cdot \nabla u (x) \left(\frac{\mu}{2}-d_F(x)\right)^+dx \ge -\dfrac{\mu}{2}\int_\Omega F(\nabla u(x))\,dx.
\end{equation}
On the other hand, the change of variable formula \eqref{change_normal} gives that
\begin{align*}
&\int_{\Omega}F_{\xi}(\nabla d_F(x)) \cdot \nabla u(x)  \left(\frac{\mu}{2}-d_F(x)\right)^+dx \\
&=\int_{\partial \Omega}F(\nu)\int_0^{\frac{\mu}{2}}\left(\frac{\mu}{2}-t\right)\dfrac{d}{dt}[ u(\phi(y,t))] J(y,t)dt\,d\mathcal{H}^{N-1}.
\end{align*}
Integrating by parts and using the fact that $J(y,0)=1$, the above integral becomes
\begin{multline*}
-\frac{\mu}{2}\int_{\partial\Omega}u(y)F(\nu)d\mathcal H^{N-1}(y) -\int_{\partial\Omega}F(\nu)
\int_0^\frac{\mu}{2} u(\phi(y,t)) \left(\dfrac{\mu}{2}-t\right)\frac{dJ}{d t}dt\;d\mathcal{H}^{N-1}(y)\\
\hfill +\int_{\de\Omega} F(\nu)\int_0^{\frac{\mu}{2}} u(\phi(y,t))J (y,t)dt\;d\mathcal H^{N-1}(y)\\
\le -\frac{\mu}{2}\int_{\partial\Omega}u(y)F(\nu)d\mathcal H^{N-1}(y)+(N-1)\int_{\de \Omega}F(\nu)\int_0^{\frac{\mu}{2}}u(y)J(y,t)\,dt\;d\mathcal{H}^{N-1}(y)\\ 
       \hfill +\int_{\de\Omega} F(\nu)\int_0^{\frac{\mu}{2}} u(\phi(y,t))J (y,t)dt\;d\mathcal H^{N-1}(y)\\
    \\=-\frac{\mu}{2}\int_{\partial\Omega}u(y)F(\nu)d\mathcal{H}^{N-1}(y)+N\int_{\Omega}u\,dx
\end{multline*}
where in the inequality we have used \eqref{laplaciano_cambio} and the bound \eqref{boundmc}.
Hence, joining with \eqref{passtracciain} it holds that
\[
\int_{\partial \Omega}u(y)F(\nu)d\mathcal{H}^{N-1}\le \abs{Du}_F(\Omega)+\dfrac{2N}{\mu}\int_{\Omega}u(x)\,dx.
\]
\end{proof}
\begin{rem}
    Using the notation of the above theorem, we explicitly observe that the constant in \eqref{trace_inequality_1} is 
    \[
    c=\frac{2N}{\mu}.
    \]
\end{rem}

\section{Interior approximation}\label{approx_sec}
Now we provide an approximation result for $BV$-functions by smooth functions with compact support in $\Omega$. 

We preliminary state two useful lemmas. Firstly, we recall from \cite[Lemma 3.2]{littig2014} the following result on diffeomorphic perturbations of sets $\Omega$ with Lipschitz boundary. We denote by $\iota$ and $I$ the identical vector and matrix function, respectively.
\begin{lem}
\label{diff}
Let $\Omega\subset \mathbb R^N$ be a bounded open set with Lipschitz boundary. Then there exists $\tau_0>0$ and, for $0\leq\tau\leq\tau_0$, a family of $C^\infty-$diffeomorphisms $\Phi^\tau : \R^N\to\R^N$ with inverses $\Psi^\tau$ such that
\begin{itemize}
    \item $\Phi^0=\Psi^0=\iota$;
    \item $\Phi^\tau\to \iota$ and $\Psi^\tau\to \iota$ as $\tau\to 0$ uniformly on $\R^N$;
    \item $\nabla\Phi^\tau (x)\to I$ and $\nabla\Psi^\tau(x)\to I$ as $\tau\to 0$ uniformly with respect to $x$ on $\R^N$;
    \item $\Phi^\tau(\overline\Omega)\Subset\Omega$ for all $\tau\in(0,\tau_0]$.
\end{itemize}
\end{lem}

Now, we give the anisotropic version of the change of coordinates formula for $BV$-functions, stated in \cite[Lemma 10.1]{giusti1984minimal}. 

\begin{lem}
\label{var}
Let $u$ be a function in $BV_\textrm{loc}(\Omega)$, $\Phi:\R^N\to\R^N$ be a diffeomorphism and $A\Subset\Omega$. Then
\begin{equation}
\label{var_thesis}
\abs{D(u\circ \Phi^{-1})}_F(\Phi(A))= \abs{HDu}_F(A),
\end{equation}
where $H=|\det \nabla\Phi|[\nabla\Phi]^{-1}$.
\end{lem}

\begin{proof}
Let us consider $u\in C^1(\Omega)$ and $g\in C_0^1(A;\R^N)$, then the following change of area formula holds
\begin{equation}
    \label{passgiusti}
    \begin{split}   
\int_{\Phi(A)}(g\circ \Phi^{-1}) \cdot \nabla (u\circ \Phi^{-1}) dx&=\int_{\Phi(A)}(g \circ  \Phi^{-1}) \cdot ((\nabla u\circ \Phi^{-1})\nabla\Phi^{-1}) dx\\
&=\int_A g \cdot (\nabla u (\nabla \Phi^{-1}\circ \Phi)) |\det \nabla\Phi| dz\\ &=\int_A g \cdot (H \nabla u) dz.
\end{split}
\end{equation}
Thus, the thesis \eqref{var_thesis} holds for $u$ in $C^1(\Omega)$, that is 
\[
\int_{\Phi(A)} F(\nabla(u\circ\Phi^{-1}))dx=
\int_{A} F(H\nabla u)dx.
\]

Suppose now that $u\in BV_{\text{loc}} (\Omega)$. By Proposition \ref{approximation_BV} we can approximate $u$ by a sequence $\{u_i\}\subset C^\infty$. Moreover, the corresponding functions $u_i\circ \Phi{^-1}$ converge to $u\circ \Phi{^-1}$ in $L^1(A)$.

Hence, we can pass to the limit in \eqref{passgiusti}, obtaining
\begin{equation}
\label{cambio_di_variabili}
\int_{\Phi(A)}(g \circ  \Phi^{-1})\cdot dD(u\circ \Phi^{-1})=
\int_A g\cdot H\, dD u=\int_{A} g\cdot (H \nu)d\abs{Du}
\end{equation}
where $\nu$ is obtained by differentiating $Du$ with respect to $|Du|$.

If $F^o(g)\leq 1$, then also $F^o(g\circ \Phi^{-1})\leq 1$ and $\spt (g\circ \Phi^{-1}))\subseteq \Phi(A)$. Therefore, by definition of total variation with respect to $F$, we have
\begin{equation}
\label{from_def}
\int_{A} g\cdot (H \nu)d\abs{Du} \le\abs{D(u\circ \Phi^{-1})}_F(\Phi(A)),
\end{equation}
The inequality \eqref{prodscal} implies that
\[ 
\sup_{F^o(g)\leq 1}\int_A g\cdot (H \nu) \ d\abs{Du}= \abs{HDu}_F(A).
\]
Hence, taking the supremum on the left hand side in \eqref{from_def}, it holds
\[
\abs{HDu}_F(A)\leq\abs{D(u\circ \Phi^{-1})}_F(\Phi(A)).
\]
For the reverse inequality, we consider $g=\gamma\circ\Phi\in C^1_0(A;\R^N)$, where $\gamma\in C^1_0(\Phi(A);\R^N)$ and $F^o(\gamma)\leq 1$. Therefore $g\circ\Phi^{-1}=\gamma$ and by \eqref{cambio_di_variabili}, we have
\[
\begin{split}
\int_{\Phi(A)}\gamma\cdot dD(u\circ \Phi^{-1})&=
\int_A ( \gamma \circ \Phi)\cdot (H \nu)d\abs{D u}(A)\\
&\leq \int_A F^o( \gamma \circ \Phi) |H| d|D u|_F\leq  \int_A \abs{H} d\abs{Du}_F.
\end{split}
\]
Hence, we have
\[
\abs{D(u\circ \Phi^{-1})}_F(\Phi(A))\leq  \abs{HDu}_F(A).
\]
\end{proof}

At this stage, we are in position to state the main approximation result.
\begin{thm}
\label{approssimazione_thm}
Let $\Omega\subset\R^N$ be an open bounded set with Lipschitz boundary  and let $u\in BV(\Omega)\cap L^p(\Omega)$ for some $p\in [1,\infty)$. Then there exists a sequence $\{u_k\}_{k\in\mathbb N}\subseteq C_0^\infty(\Omega)$ such that, for any $q\in [1,p]$,
\[
u_k\to u\ \ \text{in}\ \ L^q(\Omega)\quad\text{and}\quad |D u_k|_F(\mathbb R^N)\to |D u|_F(\mathbb R^N).
\]
\end{thm}

\begin{proof}
Let us fix a family $(\Phi^\tau)_{0 \leq \tau\leq\tau_0}$ of diffeomorphisms fron $\R^N$ to $\R^N$ with inverses $(\Psi^\tau)_{0 \leq \tau\leq\tau_0}$ according to Lemma \ref{diff}, and consider 
\[
u^\tau\coloneqq u \circ \Psi^\tau \quad \textrm{for} \quad \tau \in [0,\tau_0].
\]
By construction $u^\tau=0$ a.e. outside a compact subset of $\Omega$. By \cite[Theorem 3.1]{littig2014}, 
we know that $u^\tau \in L^p(\Omega)$ for all $\tau \in [0,\tau_0]$, $u^\tau\to u$ in $L^p(\Omega)$ and also in $L^p(\R^N)$.\\ The change of coordinates formula in Lemma \ref{var} implies that \[
\abs{Du^\tau}_F(\R^N)=\int_{\R^N}\abs{(\nabla\Psi^\tau)^T
}
\;\abs{\det(\nabla\Phi^\tau)}\;d\abs{Du}_F.
\]
Thus $u^\tau \in BV(\R^N)$ and, since the integrand on the right hand side uniformly converges to $1$, 
\[
\lim_{\tau \rightarrow 0^+}\abs{Du^\tau}_F(\R^N)=\abs{Du}_F(\R^N)
\]
It remains only to prove that there exists $v^\tau\in C^\infty_0(\Omega)$ such that 
\[
||u^\tau-v^\tau||_p<\tau\quad\text{and}\quad
\abs{\abs{Du^\tau}_F(\R^N)-\abs{Dv^\tau}_F(\R^N)}<\tau.
\]
The first convergence (in $L^p$) has been proved in \cite[Theorem 3.1]{littig2014}; meanwhile the convergence of the total variation is based on the following argument.

For any $\varepsilon>0$, let us consider the mollification $u_\varepsilon:=u^\tau*\eta_\varepsilon$, where $\eta_\varepsilon (x)=:\varepsilon^{-n}\eta(\varepsilon^{-1}x)$, for the standard mollifier $\eta$. Hence $u_\varepsilon \to u^\tau$ in $L^p(\Omega)$ and for the zero extensions, in $L^1(\R^N)$ \cite[Proposition 3.2.c]{ambrosio2000functions}.

Then, by the lower semicontinuity of the anisotropic total variation, we have
\[
\abs{Du^\tau}_F(\R^N) \leq \lim_{\varepsilon \rightarrow 0^+}\inf \abs{Du_\varepsilon}_F(\R^N).
\]
Therefore, it remains to prove the opposite inequality 
\begin{equation}
\label{dis_anis_alto}
\lim_{\varepsilon \rightarrow 0^+}\sup\abs{Du_\varepsilon}_F(\R^N) \leq \abs{Du^\tau}_F(\R^N).
\end{equation}

Let us choose $\varphi \in C^{\infty}_0(\R^N,\R^N)$ with $F^o(\varphi)\leq 1$ and calculate 
\begin{equation}
\label{diseg_alto}
\int_{\R^N}u_\varepsilon \ast \dive ( \varphi)dx\int_{\R^N}u^\tau(\eta_\varepsilon \ast \dive \varphi)dx=\int_{\R^N}u^\tau\dive(\eta_\varepsilon \ast  \varphi)dx\leq \abs{Du^\tau}_F(\R^N),
\end{equation}
where the inequality in the last term holds since $F^o(\eta_\varepsilon \ast \varphi)\le 1$. Indeed, by using the $1$-homogeneity of $F^o$ and Jensen's Inequality (see, for instance, \cite[Lemma 1.8.2]{Jensen}), we gain that
\[
F^o\left(\int_{\R^N} \eta_{\epsilon}(x-y)\varphi(y)dy\right)
\leq \int_{\R^N} F^o(\eta_{\epsilon}(x-y)\varphi(y))dy=\int_{\R^N} \eta_{\varepsilon}(x-y)F^o(\varphi(y))dy\le 1.
\]
Hence, by passing to the limit in \eqref{dis_anis_alto}, we reach the inequality \eqref{diseg_alto} by the arbitrariness of $\varphi$.
\end{proof}

\section{The first Robin eigenvalue of the Finsler $p$-Laplacian as $p\to 1$}
\label{eig_sec}
In this Section, we give an application of the results proved above to a Robin eigenvalue problem. More precisely, our aim is analyze the $\Gamma$-limit of the functional 
\begin{equation}
\label{functional_Jp}
J_p(\varphi)=\dfrac{\ds \int_\Omega F^p(\nabla\varphi)dx+\beta\ds\int_{\partial \Omega} \abs{\varphi}^pF(\nu)d\mathcal{H}^{N-1}}{\ds\int_\Omega \abs{\varphi}^pdx},\quad \varphi\in W^{1,p}(\Omega)\setminus\{0\},
\end{equation}
where $\Omega$ is a bounded, connected, sufficiently smooth open set, $p>1$ and $\beta\in\R$, and prove an isoperimetric inequality for the limit, as $p\to 1^+$, of the first eigenvalue
\begin{equation}
\label{lambda_p}
\lambda_1(\Omega,p,\beta)=\inf_{\substack{\varphi\in W^{1,p}(\Omega)\\ \varphi \ne0}}J_p(\varphi),
\end{equation}
 depending on the value of the parameter $\beta$. A key point for proving this result is the convergence of the functional $J_p$.



We first recall the following existence result for \eqref{lambda_p} holds.
\begin{thm}[\cite{fragavitone,della2022sharp}]
\label{existence}
Let $p>1$, $\beta\in\R$ and $\Omega$ bounded Lipschitz domain. Then there exists a minimum $u \in C^{1,\alpha}(\Omega)\cap C(\overline{\Omega})$ of (\ref{lambda_p}) that satisfies
\begin{equation}
\begin{cases}
-\mathcal Q_p u=\lambda(\Omega,p,\beta)\abs{u}^{p-2}u \;&\textrm{in}\;\Omega\\
F^{p-1}(\nabla u) F_\xi(\nabla u)\cdot \nu +\beta F(\nu)\abs{u}^{p-2}u=0\;&\textrm{on}\;\partial \Omega.
    \end{cases}
\end{equation}
Moreover, $u$ does not change sign in $\Omega$. Finally, $\lambda_1(\Omega,p,\beta)$ is positive if $\beta>0$, while is negative if $\beta<0$.
\end{thm}

\subsection{The case $p=1$}
In order to study the limit case of $J_p$ as $p$ goes to $1$, we consider the functional
\begin{equation}
    \label{functional_J}
J(\varphi)=\dfrac{\abs{D\varphi}_F(\Omega)+\min\{\beta,1\}\ds \int_{\partial \Omega}\abs{\varphi}F(\nu)d\mathcal H^{N-1}}{\ds\int_\Omega \abs{\varphi}dx},
\end{equation}
where $\varphi\in BV(\Omega)$ and $u\not \equiv0$. Hence, we study the associated minimum problem
\begin{equation}
    \label{Lambda}
\Lambda(\Omega,\beta)=\inf_{\substack{\varphi \in BV(\Omega)\\ \varphi \not\equiv 0}}J(\varphi).
\end{equation}

Depending on $\beta$, we will impose different assumptions on the regularity of the domain.  Indeed:
\begin{itemize}
    \item if $\beta\ge 0$, we will suppose that $\partial \Omega$ is Lipschitz;
    \item if $-1<\beta<0$, we will assume that $\partial \Omega$ is $C^2$.
\end{itemize}
 In particular, this difference depends on the fact that in the case $\beta<0$ 
 we use the trace inequality, studied in Section \ref{trace_sec}.

Finally, if $\beta\le -1$ the problem is not well posed; indeed if $\beta<-1$, then $\Lambda(\Omega,\beta)=-\infty$ while if $\beta=-1$, $\Lambda$ is finite but can be not achieved, also in the case of smooth domains. For further details, we refer the reader to the Euclidean case treated in \cite{della2022behavior}.

Let us discuss the presence of the term $\min\{\beta,1\}$ in \eqref{functional_J}. For any value of $\beta$, it could seem more natural to study the problem
\begin{equation}
\label{lambda_1}
\lambda(\Omega,1,\beta)=\inf_{\substack{\varphi \in BV(\Omega)\\ \varphi\neq 0}} \dfrac{\abs{D\varphi}_F(\Omega)+\beta\ds \int_{\partial \Omega}\abs{\varphi}F(\nu)d\mathcal H^{N-1}}{\ds\int_\Omega \abs{\varphi}dx}.
\end{equation}
Actually, we have that for $\beta\ge 1$ it holds
\[
\lambda(\Omega,1,\beta)=\Lambda(\Omega,\beta)=h_F(\Omega),
\]
where $h_F(\Omega)$ is the first Cheeger constant of $\Omega$ in the Finsler setting (see e.g. \cite{BoundaryTrace}): 
\begin{equation}
\label{cheeger_def}
h_F(\Omega)=\inf_{\substack{\varphi \in BV(\Omega)\\ \varphi \not\equiv 0}}\dfrac{\abs{D\varphi}_F(\R^N)}{\ds\int_\Omega \abs{\varphi}dx}=\inf_{\substack{E\subseteq \Omega}}\dfrac{P_F(E;\Omega)}{\ds\abs{E}}.
\end{equation}
Indeed, in this case, it is immediate to see that 
\[
\lambda(\Omega,1,\beta)\geq h_F(\Omega).
\]
On the other hand, if $u$ is a minimizer of (\ref{lambda_1}), then, by Theorem \ref{approssimazione_thm}, there exists $u_k \in C^{\infty}_0 (\Omega)$ such that 
\[
u_k \xrightarrow{L^q}u, \qquad\qquad \abs{\abs{\nabla u_k}}_{L^1(\Omega)}\xrightarrow{L^1} \abs{D u}_F(\R^N),
\]
for any $q \leq \dfrac{N}{N-1}$. Therefore
\[
\lambda (\Omega, 1, \beta)\leq    \lim_{k\rightarrow+\infty} \dfrac{\ds\int_\Omega F(\nabla u_k)dx}{\ds\int_\Omega \abs{u_k}dx}=h_F(\Omega).
\]

Now we focus on the possibility of studying the minimization problem (\ref{Lambda}) restricting our analysis to characteristic functions. Hence if $E\subseteq \Omega$, we have
\begin{align*}
J(\chi_E)
=\dfrac{P_F(E;\Omega)+\min\{1,\beta\}\ds\int_{\partial \Omega\cap \partial^* E}F(\nu_E)d\mathcal{H}^{N-1}}{\abs{E}}.
\end{align*}
By denoting 
\[
R(E,\beta):=J(\chi_E),
\]
we consider the minimization problem  
\begin{equation}
    \label{ell_problem}
\ell(\Omega,\beta)=\inf_{E\subseteq \Omega}R(E,\beta).
\end{equation}
Before proving the equivalence between problems \eqref{Lambda} and \eqref{ell_problem}, we need the following result on the lower semicontinuity of the numerator of the functional $J$.
\begin{lem}\label{prop_modica_salto}
Let $\beta \ge -1$. The functional
\[
G(u)=|Du|_F(\Omega)+\min\{1,\beta\}\int_{\partial\Omega}|u|F(\nu)d\mathcal H^{N-1}
\]
is lower semicontinuous on $BV(\Omega)$ with respect to the topology of $L^1(\Omega)$.
\end{lem}
\begin{proof}
If $\beta\ge 0$, the lower semicontinuity (with $\Omega$ Lipschitz) follows immediately by the lower semicontinuity of each term. Then we assume $\beta <0$ (and $\Omega$ in $C^2$). The proof is an adaptation of \cite[Proposition 1.2]{modica1987gradient} to the Finsler case.

Let $u\in BV(\Omega)$, and let us consider a sequence $\{u_k \}_{k\in\N}\subseteq BV(\Omega)$ converging to $u$ in $L^1(\Omega)$; we have the following estimate
\begin{equation}
    \label{modica_e1}
G(u)-G(u_k)\leq |Du|_F(\Omega)-|Du_k|_F(\Omega)+\int_{\partial\Omega}|u-u_k|F(\nu)d\mathcal H^{N-1}.
\end{equation}
Now, for a fixed $\delta>0$, let us define $\Omega_\delta=\{x\in\Omega :  d_\mathcal E(x)< \delta\}$, where $d_\mathcal E$ is the standard Euclidean distance to the boundary of $\Omega$ ; moreover let us consider  $v_\delta=(1-\chi_\delta)(u-u_k)$, where $\chi_\delta$ is a cut-off function such that $\chi_\delta=1$ in $\Omega\setminus\Omega_\delta$ and $|\nabla\chi_\delta|\le \frac{2}{\delta}$ in $\Omega$. The trace inequality \eqref{trace_inequality_1} applied to $v_\delta$ gives
\begin{equation}
    \label{modica_e2}
\int_{\partial\Omega}|u-u_k|F(\nu)d\mathcal H^{N-1}\leq \abs{D (u-u_k)}_F(\Omega_\delta)+\frac{2b}{\delta}\int_{\Omega_\delta}|u-u_k|dx+c\int_{\Omega_\delta}|u-u_k|dx.
\end{equation}

Moreover, we have
\begin{equation}
    \label{modica_e3}
|D (u-u_k)|_F(\Omega_\delta)\le |D u|_F(\Omega_\delta)+|D u_k|_F(\Omega_\delta)+|D (u-u_k)|_F(\partial(\Omega\setminus\Omega_\delta)),
\end{equation}
but last term is zero on a set of $\delta$'s of positive measure because $u-u_k\in BV(\Omega)$, for all $k\in\N$. Hence, by \eqref{modica_e1}-\eqref{modica_e2}-\eqref{modica_e3}, we gain:
\[
G(u)-G(u_k)\leq |Du|_F(\Omega)+|D u|_F(\Omega_\delta)-|D u_k|_F(\Omega\setminus\Omega_\delta)+\left(\frac{2b}{\delta}+c\right)\int_{\Omega_\delta}|u-u_k|dx.
\]
By the lower semicontinuity of the functional $|D u_k|_F(\Omega\setminus\Omega_\delta)$ in $L^1(\Omega\setminus\Omega_\delta)$, we have that
\[
\limsup_{k\to+\infty}[ G(u)-G(u_k)]\le 2 |Du|_F(\Omega_\delta).
\]
The conclusion follows by sending $\delta\to 0^+$.
\end{proof}
At this stage, we state the main existence result of the minimum problem \eqref{Lambda}.
\begin{thm}\label{Lambda=ell_thm}
For any $\beta>-1$, there exists a minimum to problem (\ref{Lambda}). In particular, it holds
\[
\Lambda(\Omega,\beta)=\ell(\Omega,\beta).
\]
Moreover, if $u\in BV(\Omega)$ is a minimum of (\ref{Lambda}), then 
\[
\Lambda(\Omega, \beta)=R(\{u>t\},\beta),
\]
for some $t \in \R$.
\end{thm}
\begin{proof}
Let $u_n$ be a minimizing sequence in $BV(\Omega)$ of \eqref{Lambda}, such that $\norm{u_n}_{L^1(\Omega)}=1$.
If $\beta>0$, then $u_n$ is bounded in $BV(\Omega)$ and hence 
\[
u_n\overset{\ast}{\rightharpoonup} u \;\textrm{in}\;BV(\Omega)\qquad\textrm{and}\qquad u_n \xrightarrow{L^1} u.
\] 
In particular, if $\beta \geq 1$, $J(u_n)=\abs{Du_n}_F(\R^N)$, by using the lower semicontinuity of the anisotropic total variation \cite{amar1994notion}, we obtain that
\[
J(u)\leq \liminf_{n}J(u_n).
\]
Hence $u$ is the minimum of the functional $J$.

If $0<\beta<1$, let $\Omega_\delta=\{x\in \Omega \ | \ d_\mathcal E(x)<\delta\}$, with $\delta>0$. We have 
\begin{align*}
\abs{Du_n}_F(\Omega)=\abs{Du_n}_F(\Omega \setminus \Omega_\delta)+\abs{Du_n}_F(\Omega_\delta)\geq\abs{Du_n}_F(\Omega \setminus \Omega_\delta)+\beta\abs{Du_n}_F(\Omega_\delta)
\end{align*}
and hence
\[
J(u_n)\geq \abs{Du_n}_F(\Omega\setminus\Omega_\delta)+\beta\left[\abs{Du_n}_F( \Omega_\delta)+\int_{\partial\Omega}\abs{u_n}F(\nu)d\mathcal H^{N-1}\right].
\]
Moreover, by the lower semicontinuity of $J$, we have
\[
\liminf_n J(u_n)\geq \abs{Du_n}_F(\Omega\setminus\Omega_\delta)+\beta\abs{Du_n}_F(\R^N\setminus (\Omega\setminus\Omega_\delta)).
\]
By using the fact that $u\in BV(\Omega)$, we obtain that
\[
\liminf_n J(u_n)\geq J(u),
\]
as $\delta\rightarrow 0$.

Now, let us take $-1<\beta<0$. It easily seen that $J(u_n)\leq C$. Using the trace inequality (\ref{trace_inequality_1}), we obtain 
\[
J(u_n)\geq (1+\beta)\abs{Du_n}_F(\Omega)+c\beta\geq c\beta
\]
and
\[
\abs{Du_n}_F(\Omega)\leq \frac{C}{1+\beta}-\frac{\beta c}{1+\beta}.
\]
Being $u_n \in BV(\Omega)$ and by the fact that the functional $J$ is lower semicontinuous (proved in Lemma \ref{prop_modica_salto}), we have that $u$ is a minimum of $J$.

Now, we want to prove last part of the Theorem. Obviously, we have 
\[
\Lambda(\Omega, \beta)\leq \ell(\Omega, \beta).
\]
To prove the reverse inequality, we take $u\in BV(\Omega)$ a minimizer of (\ref{Lambda}). By using the coarea formula 
\[
\abs{D u}_F(\Omega)=\int_{-\infty}^{+\infty}P_F(\{u>t\},\Omega)dt,
\]
we have 
\begin{align*}
\Lambda_F(\Omega,\beta)&=\dfrac{\ds\int_{-\infty}^{+\infty}P_F(\{u>t\},\Omega)dt+\min\{\beta,1\}\int_{-\infty}^{+\infty}\mathcal{H}^{N-1}(\partial \Omega \cap \partial\{u>t\})F(\nu)dt}{\ds\int_{-\infty}^{+\infty}\abs{\{u>t\}}dt}\\
&=\dfrac{\ds\int_{-\infty}^{+\infty}R(\{u>t\},\beta)\abs{\{u>t\}}dt}{\ds\int_{-\infty}^{+\infty}\abs{\{u>t\}}dt}\\&\geq \inf_{E\subseteq \Omega}R(E,\beta)\\&=\ell(\Omega,\beta).
\end{align*}
This shows that $\Lambda(\Omega,\beta)=\ell(\Omega,\beta)$ and, in particular, we have that 
\[
\int_{-\infty}^{+\infty}\{R(\{u>t\},\beta)-\ell(\Omega,\beta)\}\abs{\{u>t\}}dt=0
\]
and using the definition of $\ell(\Omega,\beta)$ we observe that the integrand is nonnegative. In particular, $u\not\equiv 0$ and we have that $\Lambda(\Omega,\beta)=R(\{u>t\},\beta)$.
\end{proof}


\subsection{$\Gamma$-convergence of $J_p$}
\label{convergence_sec}
Now we will prove that the functional $J_p$ $\Gamma-$converges to the functional $J$, as $p\to 1^+$.

\begin{defn}
A functional $J_p$ $\Gamma$-converges to $J$ as $p\to 1^+$ in the weak$^*$ topology of $BV(\Omega)$ if, for any $u\in BV(\Omega)$, the following hold:
\begin{enumerate}
\item[(i)]  For any sequence $u_p\in BV(\Omega)$ which converges to $u$ weak$^*$ in $BV(\Omega)$ as $p\to 1^+$, then
\begin{equation}
    \label{liminf_J}
\liminf_{p\to1^+} J_p(u_p)\geq J(u).
\end{equation}
\item[(ii)] There exists a sequence $u_p\in W^{1,p}(\Omega)$ which converges to $u$ weak$^*$ in $BV(\Omega)$ as $p\to 1^+$, such that
\begin{equation}
    \label{limsup_J}
\limsup_{p\to1^+} J_p(u_p)\leq J(u).
\end{equation}
\end{enumerate}
\end{defn}
Now, we are in position to prove the convergence theorem for the functional $J_p$.
\begin{thm}
\label{convergence_J}
Let $\beta>-1$, then $J_p$ $\Gamma$-converges to $J$ as $p\to 1^+$.
\end{thm}
\begin{proof}
Let us suppose $u_p\in W^{1,p}(\Omega)$ and $||u_p||_{L^p(\Omega)}=1$.
We give the proof by distinguishing the possible values of $\beta$. In any cases, we will have to prove \eqref{liminf_J} and \eqref{limsup_J}.

\textit{The case $\beta \ge 1$.} 
Let us fix a sequence $u_p\in W^{1,p}(\Omega)$ weak$^*$ converging to $u$ in $BV(\Omega)$, as $p\to 1^+$.
By using the H\"older-type inequality contained for example in \cite[Proposition A.1]{dos}, we have:
\[
\begin{split}
&\left(\int_{\Omega}F(\nabla u_p)dx+\int_{\partial\Omega}|u_p| F(\nu) d\mathcal H^{N-1}\right)^p \\
&\qquad 
\le 
\left(\int_{\Omega}F(\nabla u_p)^pdx+\int_{\partial\Omega}|u_p|^p F(\nu) d\mathcal H^{N-1}\right)\left(|\Omega|+P_F(\Omega)\right)^{p-1}.
\end{split}
\]
Hence we have
\[
\liminf_{p\to 1^+} J_p(u_p)\ge
|D u|_F(\Omega)+\int_{\partial\Omega}|u| F(\nu)d\mathcal H^{N-1}=|Du|_F(\R^N)=J(u).
\]
This proves \eqref{liminf_J}. 

To give the proof of \eqref{limsup_J}, we observe that, by Theorem \ref{approssimazione_thm}, there exists a sequence $\{u_k\}_{k\in\N}\subseteq C_0^\infty(\Omega)$ such that, for any $q\in [1,p]$, $u_k$ converges to $u$ in $L^q(\Omega)$ and $||F(\nabla u_k)||_{L^1(\Omega)}$ converges to $|Du|_F(\R^N)$ as $k\to +\infty$. 
Moreover, it easily seen that  that $||F(\nabla u_k)||_{L^p(\Omega)}$ converges to $||F(\nabla u_k)||_{L^1(\Omega)}$ and hence we have that there exists a subsequence $p_k \to 1^+$, as $k\to+\infty$, such that $||F(D u_k)||^{p_k}_{L^{p_k}(\Omega)}$ converges to $|Du|_F(\R^N)$ as $k\to+\infty$. This implies that $\limsup_{k\to+\infty}J_{p_k}(u_k)\ge J(u)$, that concludes the proof of \eqref{limsup_J}.

\textit{The case $0\leq\beta<1$.} Let us consider a sequence $u_p$ weak$^*$ converging to $u$ in $BV(\Omega)$, as $p\to 1^+$. 
A simple application of the Young inequality $a^p\geq pab-(p-1)b^{\frac p {p-1}}$ with $b=\frac 1p$, yields to
\[
\begin{split}
J_p(u_p)&=\int_\Omega F^p(\nabla u_p) dx+\beta \int_{\partial\Omega} |u_p|^pF(\nu) d\mathcal H^{N-1}\\
&\geq \int_\Omega F(\nabla u_p) dx+\beta \int_{\partial\Omega} |u_p|F(\nu) d\mathcal H^{N-1}-\frac{p-1}{p}\left(|\Omega|+P_F(\Omega)\right),
\end{split}
\]
Therefore, the conclusion \eqref{liminf_J} follows by applying the Proposition
\ref{prop_modica_salto}.

In order to get the second claim, Proposition \ref{approximation_BV} assures the existence of a sequence $u_k \in C^\infty(\Omega)$ strongly converging to $u$ in $L^1(\Omega)$ and $\|F(D u_k)\|_{L^1(\Omega)}$ converges to $|Du|_F(\Omega)$, as $k\to+\infty$. Moreover $u_k F(\nu)$ converges to $u F(\nu)$  in $L^1(\partial\Omega,\mathcal H^{N-1})$. An argument similar to the previous case leads us to say that $\|F(D u_k)\|_{L^{p_k}(\Omega)}$ converges to $|Du|_F (\Omega)$ and $\ds\int_{\partial\Omega}|u_k|^p F(\nu)d\mathcal{H}^{N-1}$ converges to $\ds\int_{\partial\Omega}|u|F(\nu)d\mathcal{H}^{N-1}$, as $k\to+\infty$.
Hence the sequence $\{u_k\}_{k\in\N}$ satisfies \eqref{limsup_J}.

\textit{The case $-1<\beta<0$.} 
Let us consider a sequence $u_p\in W^{1,p}(\Omega)$ weak$^*$ converging to $u$ in $BV(\Omega)$, as $p\to 1^+$.

For any $\delta>0$, let us set $\Omega_\delta=\{x\in\Omega :  d_\mathcal E(x)< \delta\}$ and consider a smooth function $\psi$ equal to zero in $\Omega\setminus\Omega_\delta$ and to one on $\partial\Omega$, such that $|\nabla\psi|\le \frac{2}{\delta}$.

The trace inequality \eqref{trace_inequality_1} applied to the function $v=(u-|u_p|^{p-1}u_p)\psi$ gives
\begin{equation}
    \label{stima_termine_bordo}
\begin{split}
&\int_{\partial\Omega}|u-|u_p|^{p-1}u_p| F(\nu)d \mathcal H^{N-1}\\
&\qquad \leq  |D(u-|u_p|^{p-1}u_p)|_F(\Omega_\delta)+\left(\frac{2b}{\delta}+c\right)\int_{\Omega_\delta} |u-|u_p|^{p-1}u_p| dx\\
& \qquad \leq  |Du|_F(\Omega_\delta)+\int_{ \Omega_\delta} F(\nabla (|u_p|^{p-1}u_p))dx+\left(\frac{2b}{\delta}+c\right)\int_{\Omega_\delta} |u-|u_p|^{p-1}u_p| dx,
\end{split}
\end{equation}
where we have used that $|D(u-|u_p|^{p-1}u_p)|_F(\partial\Omega_\delta)=0$ for a set of $\delta$'s of positive measure because $u-|u_p|^{p-1}u_p \in BV(\Omega)$.

By using \eqref{stima_termine_bordo}, we have
\begin{multline}
    \label{diff_J}
J(u)-J_p(u_p) =|Du|_F(\Omega)-\int_\Omega F^p(\nabla u_p)dx+\beta\int_{\partial\Omega} (|u|-|u_p|^{p-1}u_p)F(\nu)d\mathcal H^{N-1}\\
\le
|Du|_F(\Omega)-\int_\Omega F^p(\nabla u_p)dx+|\beta||Du|_F(\Omega_\delta) \\
\qquad+|\beta|\int_{\Omega_\delta} F(\nabla(|u_p|^{p-1}u_p))dx+|\beta|\left(\frac{2b}{\delta}+c\right)\int_{\Omega_\delta} |u-|u_p|^{p-1}u_p| dx:=A.
\end{multline}
Since $\frac{1}{|\beta|}>1$, we have
\begin{equation}
\label{stima_A}
    \begin{split}
A
&\leq 2 |Du|_F(\Omega_\delta) +|Du|_F(\Omega\setminus\Omega_\delta)-\int_\Omega F(\nabla u_p)^pdx+
\int_{\Omega\setminus\Omega_\delta} F(\nabla(|u_p|^{p-1}u_p))dx\\
&\qquad+\left(\frac{K}{\delta}+c\right)\int_{\Omega_\delta} |u-|u_p|^{p-1}u_p| dx.
\end{split}
\end{equation}
The Young inequality gives that
\begin{multline}
    \label{4.6}
\int_\Omega F(\nabla (|u_p|^{p-1}u_p))dx=\int_\Omega p|u_p|^{p-1}F(\nabla u_p)dx\\ \le \int_\Omega F^p(\nabla  u_p) dx+(p-1)\int_\Omega |u_p|^p dx.
\end{multline}
Furthermore by \eqref{diff_J}, \eqref{stima_A} and \eqref{4.6}, we have that
\begin{multline*}
J(u)-J_p(u_p)\leq 2 |Du|_F(\Omega_\delta) +|Du|_F(\Omega\setminus\Omega_\delta)-
\int_{\Omega\setminus\Omega_\delta} F(\nabla(|u_p|^{p-1}u_p))dx\\ +(p-1)\int_\Omega |u_p|^p dx
+\left(\frac{K}{\delta}+c\right)\int_{\Omega_\delta} |u-|u_p|^{p-1}u_p| dx.
\end{multline*}
Since $u_p$ converges to $u$ in $L^q(\Omega)$, then $|u_p|^{p-1}u_p$ converges to $u$ in $L^1(\Omega)$, as $p\to 1^+$. Hence, by taking $p\to 1^+$, we have
\[
\limsup_{p\to 1^+} \left[J(u)-J_p(u_p)\right]\leq 2 |Du|_F(\Omega_\delta).
\]
By sending $\delta\to 0^+$, we obtain \eqref{liminf_J}.

Finally, the inequality \eqref{limsup_J} is obtained as in the previous case.
\end{proof}
The proof of the $\Gamma$-convergence of the functional $J_p$ is useful to prove the convergence of the eigenvalues and eigenfunction, as $p\to 1^+$.
\begin{prop}
\label{convergence}
For any $\beta>-1$, it holds
\[
\lim_{p\to 1^+}\lambda_1(\Omega,p,\beta)=\Lambda(\Omega,\beta).
\]
Moreover, the minimizers $u_p\in W^{1,p}(\Omega)$ of \eqref{functional_Jp}, with $\|u\|_{L^p(\Omega)}=1$, weak$^*$ converge to a minimizer $u\in BV(\Omega)$ of \eqref{functional_J} as $p\to 1^+$.
\end{prop}
\begin{proof}
The Theorem \ref{convergence_J} assures the existence of a sequence $w_p$ converging to a fixed minimizer $\bar u$ of \eqref{functional_J}. Let us consider the sequence of minimizers $u_p$ of \eqref{functional_J}, we have:
\begin{equation}
    \label{J_upper_bounded}
\limsup_{p\to 1^+} J_p(u_p)\leq \limsup_{p\to 1^+}J_p(w_p)\leq J(\bar u)=\Lambda(\Omega,\beta).
\end{equation}
This means that $J_p(u_p)$ is upper bounded for any $p>1$.
If $\beta<0$, by the trace inequality \eqref{trace_inequality_1}, we have that
\[
\int_\Omega F^p(\nabla u_p) dx\leq \Lambda(\Omega,\beta)-\beta|D(u_p^p)|_F(\Omega)-\beta c,
\]
and, by \eqref{4.6} and \eqref{eq:omo}, that
\[
(1+\beta)a^p\int_\Omega |\nabla u_p|^p dx\leq \Lambda(\Omega,\beta)-\beta c-\beta(p-1).
\]
Hence, by the compactness, we have that $u_p$ is upper bounded in $BV(\Omega)$. If $\beta \ge 0$, this directly follows  from \eqref{J_upper_bounded}.
Therefore $u_p$ weak$^*$ converges to $u$ in $BV(\Omega)$.

Finally, by \eqref{liminf_J} and \eqref{J_upper_bounded}, we have that
\[
\Lambda(\Omega,\beta)=J(\bar u)\leq J(u)\leq \liminf_{p\to 1^+}J_p(u_p)\leq\limsup_{p\to 1^+}J_p(u_p) \leq \Lambda(\Omega,\beta),
\]
and hence the conclusion by observing that $\lambda_1(\Omega,p,\beta)=J_p(u_p)$.
\end{proof}

\subsection{An isoperimetric inequality}
\label{isoperimetric_subsec}
Here we treat the shape optimization problem for $\Lambda(\Omega,\beta)$. To this aim, we briefly recall the properties of the eigenvalue problem $\lambda_1(\mathcal W_R,p,\beta)$ and then we prove an explicit computation for $\Lambda(\mathcal W_R,\beta)$. By Theorem \ref{existence}, a minimizer of \eqref{lambda_p}
solves the following problem:
\begin{equation}
\label{anisotropic}
\begin{cases}
-\mathcal{Q}_p u=\lambda_1(\mathcal{W}_R,p,\beta)\abs{u}^{p-2}u & \textrm{in}\;\mathcal{W}_R\\
(F(\nabla u))^{p-1}F_{\xi}(\nabla u)\cdot \nu +\beta F(\nu)\abs{u}^{p-2}u=0 &\textrm{on}\;\partial\mathcal{W}_R.
\end{cases}
\end{equation}
In particular, the following result holds (refer to in \cite{fragavitone} for the positive values of the Robin parameter).
\begin{thm}
If $u_p \in C^{1,\alpha}(\mathcal W_R)\cap C(\overline{\mathcal W_R})$ is a positive solution of \eqref{anisotropic}, then there exists a monotone function $\varphi_p=\varphi_p(r)$, $r\in [0,R]$, such that $\varphi_p \in C^{\infty}(0,R)\cap C^1([0,R])$, and 
\begin{equation}
    \label{radial}
\begin{cases}
u_p(x)=\varphi_p(F^o(x)) \qquad\qquad \text{in }\overline{\mathcal{W}}_R\\
\varphi^{'}_p(0)=0\\
|\varphi^{'}_p(R)|^{p-2}\varphi'_p(R) +\beta\varphi_p(R)^{p-1}=0.
\end{cases}
\end{equation}
Moreover, $\varphi_p$ is decreasing if $\beta>0$, while is increasing if $\beta<0$.
\end{thm}

We first compute $\Lambda(\mathcal W_R,\beta)$.
\begin{prop}\label{explicit_prop}
If $\beta>-1$, then 
\begin{equation}
\label{explicit}
\Lambda(\mathcal{W}_R,\beta)=\hat{\beta}h_F(\mathcal{W}_R)=\hat{\beta}\dfrac{N}{R},
\end{equation}
where $\hat{\beta}=\min\{\beta,1\}.$
\end{prop}
\begin{proof}
If $\beta\geq 0$, we recall that 
\[\Lambda(\mathcal{W}_R,\beta)=\inf_{E\subseteq \mathcal{W}_R}J(\chi_E).
\]
By using Theorem \ref{Lambda=ell_thm} and the isoperimetric inequality, we have 
\begin{align*}
R(E,\beta)&=J(\chi_E)\\
&=\dfrac{P_F(E,\mathcal{W}_R)+\hat{\beta}\ds\int_{\partial \mathcal{W}_R\cap\partial E}F(\nu_E)d\mathcal{H}^{N-1}}{\abs{E}}\\
&\geq\hat{\beta}\dfrac{P_F(E)}{\abs{E}}\geq\hat{\beta}\dfrac{P_F(\mathcal{W}_r)}{\abs{\mathcal{W}_r}}\geq\hat{\beta}\dfrac{P_F(\mathcal{W}_R)}{\abs{\mathcal{W}_R}}\\&=\hat{\beta}\dfrac{N}{R},
\end{align*}
where $\mathcal{W}_r$ is the wulff shape of radius $r<R$, with $\abs{\mathcal{W}_r}=\abs{E}$.

This proves $\Lambda(\mathcal{W}_R,\beta)\ge\hat{\beta}\frac{N}{R}$. For the reverse inequality we take $E=\mathcal{W}_R$, hence
\[
\Lambda(\mathcal{W}_R,\beta)=\ell(\mathcal{W}_R,\beta)\leq R(\mathcal{W}_R,\beta)=\hat{\beta}\dfrac{N}{R}.
\]
Now, we study the case $-1<\beta<0$ and we will make use of the $\Gamma$-convergence.\\ Hence, let $u_p\in W^{1,p}(\mathcal{W}_R)$ a minimizer of (\ref{lambda_p}). We know, thanks to the Proposition \ref{convergence}, that 
\[
\lim_{p\rightarrow1^{+}}\lambda_{1}(\mathcal{W}_R,p,\beta)=\Lambda(\mathcal{W}_R,\beta)
\]
and we take $u_p=\varphi_p$ as in (\ref{radial}). So, the minimizer converges strongly in $L^1(\mathcal{W}_R)$ to $u\in BV(\mathcal{W}_R)$, almost everywhere in $\mathcal{W}_R$ and $u_p\overset{\ast}{\rightharpoonup} u \;\textrm{in}\;BV(\Omega)$ for $p\rightarrow 1^+$. Moreover, $u_p$ is radially increasing, hence $u$ is nondecreasing and this implies that its superlevel sets $\{u>t\}$ are concentric Wulff shapes $\{r<F^o(x)<R\}$ and, by Theorem \ref{Lambda=ell_thm}, it holds that 
\[
\Lambda(\mathcal{W}_R,\beta)=\dfrac{N}{R}\dfrac{(\frac{r}{R})^{N-1}+\beta}{1-(\frac{r}{R})^{N-1}}
\]
for some $r\in [0,R[$. Therefore, by minimizing the function 
\[
f(t)=\dfrac{t^{N-1}+\beta}{1-t^N}\qquad t\in [0,1[,
\]
we observe that the minimum is attained at $t=0$. Hence, the thesis follows.
\end{proof}
Finally, we prove an isoperimetric inequality for $\Lambda(\Omega,\beta)$ when a volume constraint holds: if $\beta\geq 0$, the Wulff shape is a minimizer and, if $\beta<0$, it is a maximizer for $\Lambda(\Omega,\beta)$.
\begin{prop}
If $\beta\geq 0$ and $\mathcal{W}_R$ is the wulff shape of radius $R$ and $\abs{\mathcal{W}_R}=\abs{\Omega}$, then 
\[
\Lambda(\mathcal{W}_R,\beta)\leq \Lambda(\Omega,\beta).
\]
If $-1<\beta<0$, then 
\[
\Lambda(\mathcal{W}_R,\beta)\geq \Lambda(\Omega,\beta).
\]
\end{prop}
\begin{proof}
If $\beta \geq 0$, by using the same argument of the Proposition \ref{explicit_prop}, we have that then
\[
R(E,\beta)=J(\chi_E)\geq\hat{\beta}\dfrac{P_F(\mathcal{W}_R)}{\abs{\mathcal{W}_R}}=\hat{\beta}\dfrac{N}{R}=\Lambda(\mathcal{W}_R,\beta),
\]
for any $E\subseteq\Omega$. The conclusion follows by passing to the infimum on the set $E\subseteq\Omega$ and using Theorem \ref{Lambda=ell_thm}.

If $-1<\beta <0$, then using the  isoperimetric inequality \cite{fonseca1991uniqueness} and \eqref{explicit}, we have that 
\[
\Lambda(\Omega,\beta)\leq \beta\dfrac{P_F(\Omega)}{\abs{\Omega}}\leq \beta \dfrac{P_F(\mathcal{W}_R)}{\abs{\mathcal{W}_R}}=\Lambda(\mathcal{W}_R,\beta).
\]
\end{proof}

\section*{Acknowledgement}

This work has been partially supported by the  MIUR-PRIN 2017 grant ``Qualitative and quantitative aspects of nonlinear PDE's'', by GNAMPA of INdAM, by  the FRA Project (Compagnia di San Paolo and Universit\`a degli studi di Napoli Federico II) \verb|000022--ALTRI_CDA_75_2021_FRA_PASSARELLI|.

%
\bibliographystyle{abbrv}

\end{document}